\documentclass[11pt]{article}

\usepackage{amssymb, amsmath, amsthm, graphicx}
\usepackage[left=1in,top=1in,right=1in]{geometry}
\date{}

\theoremstyle{plain}
      \newtheorem{theorem}{Theorem}[section]
      \newtheorem{lemma}[theorem]{Lemma}
            \newtheorem{claim}[theorem]{Claim}

      \newtheorem{corollary}[theorem]{Corollary}

\theoremstyle{definition}

\theoremstyle{remark}

\title{Enumeration of intersection graphs of $x$-monotone curves\footnote{A preliminary version of this paper appeared in the proceedings of the 32nd International Symposium on Graph Drawing and Network Visualization (GD 2024).}}

\begin{document}

\author{Jacob Fox\thanks{Stanford University, Stanford, CA. Supported by NSF awards DMS-2452737 and DMS-2154129. Email: {\tt jacobfox@stanford.edu.}} \and J\'anos Pach\thanks{R\'enyi Institute of Mathematics, Budapest, Hungary. Supported by NKFIH grants K-131529, Austrian Science Fund Z 342-N31, and ERC Advanced Grant 882971 ``GeoScape.'' Email:
{\tt pach@renyi.hu}.}\and  Andrew Suk\thanks{Department of Mathematics, University of California at San Diego, La Jolla, CA, 92093 USA. Supported by NSF CAREER award DMS-1800746, NSF grant DMS-1952786, and NSF grant DMS-2246847. Email: {\tt asuk@ucsd.edu}.} }

\maketitle

\begin{abstract}
\noindent A curve in the plane is \emph{$x$-monotone} if every vertical line intersects it at most once. A family of curves are called \emph{pseudo-segments} if every pair of them have at most one point in common. We construct $2^{\Omega(n^{4/3})}$ families, each consisting of $n$ labelled $x$-monotone pseudo-segments such that their intersection graphs are different. On the other hand, we show that the number of such intersection graphs is at most $2^{O(n^{4/3}\log^2n)}$. Our proof uses a new upper bound on the number of set systems of size $m$ on a ground set of size $n$, with VC-dimension at most $d$. Much better upper bounds are obtained if we only count \emph{bipartite} intersection graphs, or, in general, intersection graphs with bounded chromatic number. 

\end{abstract}

\section{Introduction}
The intersection graph of a collection $\mathcal{C}$ of sets has vertex set $\mathcal{C}$ and two sets in $\mathcal{C}$ are adjacent if and
only if they have nonempty intersection. 
 A \emph{curve} is a subset of the plane which is homeomorphic to
the interval $[0, 1]$. A \emph{string graph} is the intersection graph of a collection of curves. It is straightforward
to show that the intersection graph of any collection of arcwise connected sets in the plane is a string graph.  A collection of curves in the plane is called a collection of \emph{pseudo-segments} if every pair of them have
at most one point in common.  Finally, we say that a curve in the plane is \emph{$x$-monotone} if every vertical line intersects it in at most one point.

For a family $\mathcal{F}$ of simple geometric objects (namely those that can be defined by semi-algebraic relations of bounded description complexity), such as segments or disks in the plane, Warren's theorem~\cite{W} can be used to show that the number of labelled graphs on $n$ vertices which can be obtained as the intersection graph of a collection of $n$ objects from $\mathcal{F}$ is $2^{O(n\log n)}$ (see \cite{PS,MM}).  Moreover, for many simple geometric objects, a result of Sauermann \cite{S} shows that these bounds are essentially tight.  Unfortunately, for general curves, Warren's theorem cannot be applied.  In this paper, we estimate the number of graphs which can be obtained as the intersection graph of curves in the plane under various constraints.

In \cite{PT06}, Pach and T\'oth showed that the number of intersection graphs of $n$ labelled pseudo-segments is at most $2^{o(n^2)}$.  This bound was later improved by Kyn\v cl \cite{K13} to $2^{O(n^{3/2}\log n)}$. It was noted in both papers that the best known lower bound on the number of intersection graphs of $n$ labelled pseudo-segments is $2^{\Omega(n\log n)}$, the number of different labellings of the vertex set.  Our first result significantly improves this bound.

\begin{theorem}\label{main0}

There are at least $2^{\Omega(n^{4/3})}$ labelled $n$-vertex intersection graphs of $x$-monotone pseudo-segments. 

\end{theorem}

\noindent In the other direction, we show that this bound is tight up to logarithmic factors in the exponent for intersection graphs of $x$-monotone pseudo-segments.

\begin{theorem}\label{main12}
There are at most $2^{O(n^{4/3}\log^2n)}$ labelled $n$-vertex intersection graphs of $x$-monotone pseudo-segments in the plane.
    \end{theorem}

In the case of small clique number, we obtain the following.

\begin{theorem}\label{main3}
There are at most $2^{O(kn\log^2 n)}$ labelled $n$-vertex intersection graphs of $x$-monotone pseudo-segments with clique number at most $k$.  Moreover, for $k < n^{1/3}$, this bound is tight up to a polylogarithmic factor in the exponent.

\end{theorem}

In \cite{PT06}, Pach and T\'oth showed that the number of string graphs on $n$ labelled vertices is $2^{\frac{3}{4}\binom{n}{2}+o(n^2)}$.  Moreover, their result holds for $x$-monotone curves.  Our next result shows that there are far fewer \emph{bipartite} intersection graphs of $x$-monotone curves in the plane.

\begin{theorem}\label{main}

There are at most $2^{O(n\log^2 n)}$ labelled $n$-vertex bipartite intersection graphs of $x$-monotone curves in the plane.

\end{theorem}

\noindent Let us remark that the $x$-monotone condition in the theorem above cannot be removed.  An interesting construction due to Keszegh and P\'alv\"olgyi \cite{KP} implies that the number of $n$-vertex \emph{bipartite} string graphs is at least $2^{\Omega(n^{4/3})}$.

For the non-bipartite case, suppose $G$ is an $n$-vertex intersection of graph of $x$-monotone curves, such that $G$ has chromatic number $q\geq 3$.  Then we can partition $V(G)$ into $q$ parts such that each part is an independent set. By further partitioning each part, arbitrarily, such that the size of each remaining part is at most $n/q$, we end up with at most $2q$ parts.  By applying Theorem \ref{main} to each pair of parts, we obtain the following corollary.

\begin{corollary}\label{main2}
There are at most $2^{O(qn\log^2 n)}$ labelled $n$-vertex intersection graphs of $x$-monotone curves with chromatic number at most $q$.

\end{corollary}

Two drawings of a graph are \emph{weakly isomorphic} if the intersection graphs of their edges (with edges labelled by their endpoints) are the same. 
A {\it topological graph} is a graph drawn in the plane with possibly intersecting edges, and it is called \emph{simple} if every pair of edges 
intersect at most once. A topological graph is \emph{$k$-quasiplanar} if it has no $k$ pairwise crossing edges with distinct endpoints. 

The above results can be used to get upper bounds on the number of non-weakly isomorphic drawings of a graph with certain properties. The next result is an immediate corollary of Theorem~\ref{main3}, combined with the theorem of Valtr~\cite{Valtr} stating that the number of edges of a $k$-quasiplanar simple topological graph on $n$ vertices with $x$-monotone edges is $O_k(n\log n)$. 

\begin{corollary}\label{drawcor}
Given any $n$-vertex graph $G$, the number of non-weakly isomorphic drawings of $G$ as a $k$-quasiplanar simple topological graph with $x$-monotone edges is 
$2^{O_k(n\log^3 n)}$. 
\end{corollary} 

With respect to Theorems \ref{main3} and \ref{main} and Corollaries \ref{main2} and \ref{drawcor}, we conjecture that one of the logarithmic factors in the exponent can be removed. (In the case of Corollary \ref{drawcor}, perhaps a factor $\log^2 n$ in the exponent can removed). We discuss what is known from below at the end of the paper, and we describe a simple construction that shows there are $2^{\Omega(n \log n)}$ unlabelled bipartite graphs on $n$ vertices that are intersection graphs of segments.   

Our paper is organized as follows.  In the next section, we prove Theorem \ref{main0}.  In Section~\ref{secvc}, we establish a bound on the number of set systems of size $m$ on a ground set of size $n$ with VC-dimension $d$. After completing this work, we learned that Alon, Moran, and Yehudayoff also found a proof of a similar result for balanced bipartite graphs (see Theorem 13 in \cite{alon2}). 
 Together with the well-known cutting lemma, we prove Theorem \ref{main12} in Section \ref{secpseudo}.  In Section \ref{secbip2}, we prove Theorem \ref{main}.  We conclude the paper with some remarks.

\section{Proof of Theorem \ref{main0}}

The proof of Theorem \ref{main0} is based on a well-known construction from incidence geometry.  We prove the following more general result.

\begin{theorem}\label{main4}
For $k \leq n^{1/3}$, there are at least $2^{\Omega(kn)}$ $n$-vertex labelled intersection graphs of $x$-monotone pseudo-segments with clique number at most $k$.

\end{theorem}

\begin{proof} Let $k$ and $n$ be integers such that $k \leq n^{1/3}$. Take $$P=\{(a,b)\in \mathbb{N}^{2}:a<n^{1/3}, b<n^{2/3}\}$$ and $$\mathcal{L}=\{a'x+b'=y:a',b'\in \mathbb{N},a'<k ,b'<n^{2/3}/2\}.$$

Then we have $|P| \leq n$ and $|\mathcal{L}| \leq kn^{2/3}/2 \leq n$, and each line in $\mathcal{L}$ is incident to at least $n^{1/3}/4$ points from $P$.  For each point $p = (a,b)$ in $P$, we replace $p$ with a very short horizontal segment $\overline{p}$ with endpoints $(a,b)$ and $(a + \epsilon, b)$.  Let $\mathcal{H}$ be the resulting set of horizontal segments.

For each line $\ell \in \mathcal{L}$,  we modify $\ell$ in a small neighborhood of each point in $P$ that is incident to $\ell$ as follows. Let $\ell: y = a'x + b'$ and $p \in \ell$. Inside the circle $C$ centered at $p$ with radius $\frac{\epsilon}{2a'}$, we modify $\ell$ so that it is a semicircle along $C$ that lies either above or below $p$. After performing this operation at each point $p$ on $\ell$, and performing a small perturbation, we obtain an $x$-monotone curve $\Tilde{\ell}$.  Moreover, any two resulting $x$-monotone curves will cross at most once.  See Figure \ref{localfig}.   Let $\mathcal{L}_x$ be the resulting set of $x$-monotone curves, and note that $\mathcal{H}\cup\mathcal{L}_x$ is a set of $x$-monotone pseudo-segments.

We now count the number of intersection graphs between $\mathcal{H}$ and $\mathcal{L}_x$.  Since each line $\ell \in \mathcal{L}$ was incident to at least $n^{1/3}/4$ points in $P$, the number of different neighborhoods that can be generated for $\Tilde{\ell}$ is $2^{\Omega(n^{1/3})}$.  Moreover, two $x$-monotone curves $\Tilde{\ell},\Tilde{\ell}' \in \mathcal{L}_x$ cross if and only if their original line configuration $\ell,\ell' \in \mathcal{L}$ have distinct slope.  Since $\mathcal{L}$ has at most $k-1$ distinct slopes, the intersection graph of $\mathcal{H}\cup\mathcal{L}_x$ has clique number at most $k$, and number of such intersection graphs we can create between $\mathcal{H}$ and $\mathcal{L}_x$ is at least $2^{\Omega(kn)}$.  This completes the proof of Theorem \ref{main4}. \end{proof}

\begin{figure}
    \centering
    \includegraphics[scale=.5]{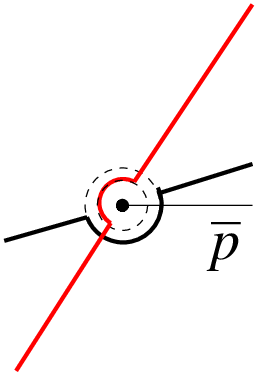}
    \caption{Modifying lines through $p$.}
    \label{localfig}
\end{figure}

\medskip

\section{Tools from VC-dimension theory}\label{secvc}

 In this section, we recall and prove results related to the notion of VC-dimension.   The {\em VC-dimension of a set system} $\mathcal{F}$ on a ground set $V$ is the {\em largest} integer $d$ for which there exists a $d$-element set $S\subset V$ such that for every subset $B\subset S$, one can find a member $A\in \mathcal{F}$ with $A\cap S=B$. Note that for a multiset system (which allows for the sets not to be necessarily distinct), the VC-dimension is the same as for the set system where we include each set that appears once.   

 Here, it will be more convenient to work with the primal shatter function of $\mathcal{F}$. The \emph{primal shatter function} of $\mathcal{F}$ is defined as $$\pi_{\mathcal{F}}(z) =  \max_{V_0 \subset V, |V_0| = z} |\{ A\cap V_0: A \in \mathcal{F}\}|.$$
 
\noindent In other words, $\pi_{\mathcal{F}}(z)$ is a function whose value at $z$ is the maximum possible number of distinct
intersections of the sets of $\mathcal{F}$ with a $z$-element subset of $V$. The primal shatter function of $\mathcal{F}^{\ast}$ is
often called the \emph{dual shatter function} of $\mathcal{F}$, and denoted as $\pi^{\ast}_{\mathcal{F}}(z) = \pi_{\mathcal{F}}(\mathcal{F}^{\ast})$.
The VC-dimension of $\mathcal{F}$ is closely related to its shatter functions. A result of Sauer and Shelah \cite{Sa,S}
states that if $\mathcal{F}$ is a set system with VC-dimensions $d$, then $\pi_{\mathcal{F}}(z)  = O(z^d)$.

Given a graph $G = (V,E)$, we define the {\em VC-dimension of} $G$ to be the VC-dimension of the set system formed by the neighborhoods of the vertices, where the neighborhood of $v\in V$ is $N(v)=\{u\in V : uv\in E\}$. In \cite{alon}, Alon et al.~proved that the number of bipartite graphs with parts of size $n$ and VC-dimension at most $d$ is at most 
$$2^{O(n^{2 - 1/d}(\log n)^{d + 2})}.$$
They further asked if the logarithmic factors can be removed.

We make progress on this question, obtaining a better bound for a more general problem, by establishing Theorem \ref{vc} below. See Alon, Moran, and Yehudayoff \cite{alon2} for a different approach to the question of \cite{alon}. Following the original proof of Alon et al.~\cite{alon}, but using the Haussler packing lemma \cite{H} (stated below) instead of Lemma~26 in \cite{alon}, one can obtain a stronger and more general bound. In addition to this, we use a different counting strategy that further removes an additional logarithmic factor. 

 


For the sake of completeness, we include the short proof below.  First, we will need some definitions. Given two sets $A,B \in \mathcal{F}$, the \emph{distance} between $A$ and $B$ is $d(A,B):=|A\bigtriangleup  B|$, where $A \bigtriangleup B = (A\cup B)\setminus (A\cap B)$ is the symmetric difference of $A$ and $B$.  We say that the set system $\mathcal{F}$ is $\delta$-separated if the distance between any two members in $\mathcal{F}$ is at least $\delta$.  The following \emph{packing lemma} was proved by Haussler in \cite{H}.

\begin{lemma}[\cite{H}]\label{haussler}

Let $\delta > 0$ and $\mathcal{F}$ be a set system on an $n$-element ground set $V$ such that $\pi_{\mathcal{F}}(z) \leq cz^d$.   If $\mathcal{F}$ is $\delta$-separated, then $|\mathcal{F}| \leq c_1(n/\delta)^{d}$ where $c_1= c_1(c,d)$.

\end{lemma}

\medskip

Let $h_{c,d}(m,n)$ denote the number of multiset systems $\mathcal{F}$ consisting of $m$ subsets of $[n]$ such that $\pi_{\mathcal{F}}(z) \leq cz^d$.   Let $h'_{c,d}(m,n)$ denote the number of set systems $\mathcal{F}$ of $m$ subsets of $[n]$ such that $\pi_{\mathcal{F}}(z) \leq cz^d$.  Clearly, $h'_{c,d}(m,n) \leq h_{c,d}(m,n)$.  It follows that $h'_{c,d}(m,n)=0$ if $m > cn^d$. Further, we can relate the two as follows. If we pick a multiset system consisting of $m$ sets whose primal shatter function is at most $cz^d$, then by throwing out repeated sets, we get a set system on the same ground set consisting of $m' \leq m$ sets. We then have to fill out these $m'$ sets to $m$ sets with repeats, including each set at least once. We thus have \begin{equation}\label{hfromh'} h_{c,d}(m,n)=\sum_{m' \leq m}h'_{c,d}(m',n){m-1 \choose m'-1}.
 \end{equation}

In what follows, $c$ and $d$ are fixed and the implicit constant in the big-O depends on $c$ and $d$. 

\begin{theorem}\label{vc}
Let $c,d \geq 2$ be fixed and $n,m \geq 2$. Then the number $h_{c,d}(m,n)$ of multiset systems $\mathcal{F}$ of $m$ subsets of $[n]$ such that $\pi_{\mathcal{F}}(z) \leq cz^d$,  satisfies $$h_{c,d}(m,n) = 2^{O(m^{1-1/d}n\log m)}.$$ Furthermore, if $m > cn^d$, then $$h_{c,d}(m,n)=2^{O(n^d\log m)}.$$
 \end{theorem}

 \begin{proof} 
Consider a linear ordering of the subsets of $[n]$. Let $\mathcal{F}$ be a multiset system of $m$ subsets of $[n]$. Let $S_1$ be the first set in $\mathcal{F}$ by the linear ordering. We will order the sets in $\mathcal{F}$ as $S_1,S_2,\ldots,S_m$ as follows. After picking $S_1,\ldots,S_{i-1}$, let $\delta_i=\max_{S \in \mathcal{F} \setminus \{S_1,\ldots,S_{i-1}\}} \min_{1 \leq j \leq i-1} d(S,S_j)$, and $S_i$ be a set $S$ that obtains the maximum, and $j_i$ be a $j$ that obtains the minimum $d(S_i,S_j)$. By our choice of the sets, 
the minimum of $d(S_a,S_b)$ over all $1 \leq a < b \leq i$ is $d(S_{j_i},S_i)$. By the Haussler packing lemma, we thus have $i = O((n/\delta_i)^d)$, or equivalently, $\delta_i = O(i^{-1/d}n)$. 

We now upper bound the number of choices of $\mathcal{F}$. There are at most $2^n$ choices of $S_1$. Each $j_i$ is a positive integer at most $i-1$, so there are at most $(m-1)! \leq m^m$ choices of $j_2,\ldots,j_m$. Having picked out this sequence of $j_i$'s, and having picked $S_1,\ldots,S_{i-1}$, we know $S_i$ must have symmetric difference at most $t_i:=O(i^{-1/d}n)$ from $S_{j_i}$. Thus, given this information, the number of choices for $S_i$ is at most $O(n^{t_i})$. Therefore, we obtain that the number of choices of $\mathcal{F}$ is at most 
\begin{eqnarray*} 2^n m^m \prod_{i=2}^m O(n^{t_i}) & \leq &  2^n m^m \prod_{i=2}^m (O(i^{1/d}))^{O(i^{-1/d}n)}\\ &= & 2^n m^m 2^{n\sum_{i=2}^m O(i^{-1/d}\log i)} \\ & = & 2^n m^m 2^{O(m^{1-1/d} n \log m)}.\end{eqnarray*} 
Note that the $2^n$ factor is at most the last factor. Hence, we obtain that the count is at most 
$m^m2^{O(m^{1-1/d} n \log m)}$. If $m \leq cn^d$, then the last factor is largest and this gives the desired bound. 

So we may assume we are in the case $m > cn^d$. In this case, by Equation (\ref{hfromh'}), the fact that $h'_{c,d}(m',n)=0$ for $m' > cn^d$ and $h'_{c,d}(m',n) \leq h_{c,d}(cn^d,n)$, we get $$h_{c,d}(m,n) \leq h_{c,d}(cn^d,n)\sum_{m' \leq cn^d}{m-1 \choose m'-1} = 2^{O(n^d \log m)}.$$
Notice that in this case, the first bound still holds, as $m^{1-1/d}n \log m \geq n^d \log m$.

 \end{proof}

\section{Intersection graphs of $x$-monotone pseudo-segments}\label{secpseudo}

In this section, we prove Theorem \ref{main12}.  We will need the following lemmas.  Recall that a \emph{pseudoline} is a two-way infinite $x$-monotone curve in the plane. An \emph{arrangement of pseudolines} is a finite collection of pseudolines such that any two members have at most one point in common, at which they cross, and each intersection point has a unique $x$-coordinate.  Given an arrangement $\mathcal{A}$ of $n$ pseudolines, we obtain a sequence of permutations of $1,\ldots, n$ by sweeping a directed vertical line across $\mathcal{A}$.  This sequence of permutations is often referred to as a \emph{partial allowable sequence of permutations}, which starts with the identity permutation $(1,\ldots, n)$, such that i) the move from one permutation to the next consists of swapping two adjacent elements, and ii) each pair of elements switch at most once.  We say that two pseudoline arrangements $\mathcal{A}_1$ and $\mathcal{A}_2$ are \emph{$x$-isomorphic} if they give rise to the same sequence of permutations, that is, a sweep with a vertical line meets the crossing pairs in the same order.

\begin{lemma}[\cite{St}]\label{arrange}
The number of arrangements of $m$ pseudolines, up to $x$-isomorphism, is $2^{\Theta(m^2\log m)}$.
    
\end{lemma}

We will also need the following result, known as the \emph{zone lemma} for pseudolines.

\begin{lemma}[\cite{bern}]\label{zone}
    Let $\mathcal{A}$ be an arrangement of $m$ pseudolines.  Then for any $\alpha \in \mathcal{A}$, the sum of the numbers of sides in all the cells in the arrangement of $\mathcal{A}$ that are supported by $\alpha$ is at most $O(m)$.
\end{lemma}

In \cite{PT06}, Pach and T\'oth proved that there is an absolute constant $d$ such that every intersection graph of a collection of pseudo-segments has VC-dimension at most $d$.  Their proof was based on hypergraph Ramsey theory, and no explicit bound for $d$ was given.  Here, we will use the following result due to Przytycki \cite{P} to bound the primal shatter function of multiset systems corresponding to intersection graphs between pseudolines and pseudo-segments.

\begin{lemma}[\cite{P}]\label{arc}
    Let $\mathcal{S}$ be a punctured sphere with Euler characteristic $\chi < 0 $.  If $p$, $q$ are punctures on $\mathcal{S}$, then the maximal cardinality of a set $\mathcal{A}$ of essential simple
arcs on $\mathcal{S}$ that are starting at $p$ and ending at $q$, and pairwise intersecting at most
once, is $|\chi|(|\chi| + 1)/2$.
\end{lemma}

We say that a collection $\mathcal{A}$ of $x$-monotone pseudo-segments in the plane is \emph{double grounded} if there are vertical lines $\ell_1$ and $\ell_2$ (called grounds) such that each curve in $\mathcal{A}$ has its left endpoint on $\ell_1$ and its right endpoint on $\ell_2$.

 \begin{lemma}\label{ptvc}
Let $\mathcal{A}$ be a collection of double-grounded $x$-monotone curves, with grounds $\ell_1$ and $\ell_2$, and let $\mathcal{B}$ be a collection of $x$-monotone curves such that $\mathcal{A}\cup\mathcal{B}$ is a collection of pseudo-segments.  Then, for the set system $\mathcal{F} = \{N(\alpha) \subset \mathcal{B}: \alpha \in \mathcal{A}\}$ with ground set $\mathcal{B}$, where $N(\alpha)$ denotes the set of curves in $\mathcal{B}$ that cross $\alpha \in \mathcal{A}$, we have $\pi_{\mathcal{F}}(z) = O(z^2)$ and $\pi^{\ast}_{\mathcal{F}}(z) = O(z^4)$.
 \end{lemma}

 \begin{proof}

  We first show that $\pi_{\mathcal{F}}(z) = O(z^2)$.  Let us consider a subset $\mathcal{B}' \subset \mathcal{B}$ of $z$ curves in $\mathcal{B}$.  Hence, $\mathcal{B}'$ corresponds to $z$ vertices in the ground set of $\mathcal{F}$.  Using the $x$-monotonicity of the curves and the assumption that $\mathcal{A}\cup\mathcal{B}$ is a collection of pseudo-segments, we know that $\alpha \in \mathcal{A}$ and $\beta \in \mathcal{B}'$ do not cross each other if and only if $\alpha$ lies below the left and right endpoints of $\beta$, or above the left and right endpoints of $\beta$.   See Figure \ref{figabex}.

\begin{figure}
\centering
\includegraphics[width=7cm]{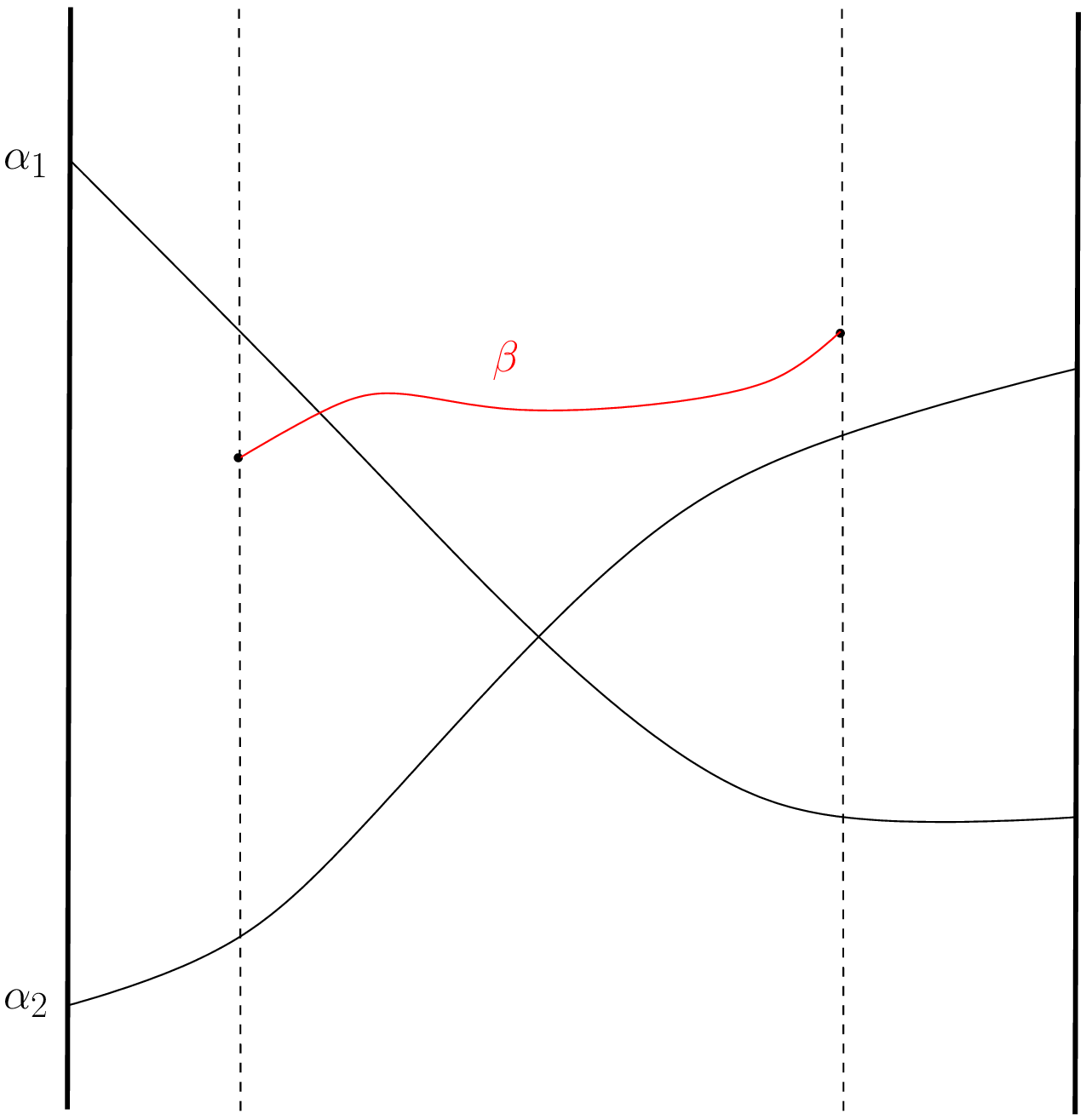}
\caption{For $\alpha_1,\alpha_2 \in \mathcal{A}$ and $\beta \in \mathcal{B}'$, $\beta$ crosses $\alpha_1$ but does not cross $\alpha_2$.}\label{figabex}
\end{figure}

Let $v$ be a point between $\ell_1$ and $\ell_2$, high above all of the curves in $\mathcal{A}\cup\mathcal{B}'$.  In order to see that $\pi_{\mathcal{F}}(z) = O(z^2)$, we can assume that the curves in $\mathcal{A}$ and $\mathcal{B}'$ lie on a sphere. By contracting the lines $\ell_1$ and $\ell_2$ into single points $p$ and $q$, resp., all curves in $\mathcal{A}$ have one endpoint at $p$ and the other at $q$. 
We turn the endpoints of the $z$ curves in $\mathcal{B}'$ into $2z$ punctures on the sphere.  Further, we place three additional punctures at $p, q$, and $v$.  Hence, every curve in $\mathcal{A}$ is an arc between $p$ and $q$.  As a result, we have a $(2z+3)$-punctured sphere $\mathcal{S}$, which has Euler characteristic $-2z - 1$. By Lemma~\ref{arc}, the maximum number of pairwise non-homotopic essential simple arcs between $p$ and $q$ on $\mathcal{S}$, pairwise crossing at most once, is at most $(z+1)(2z+1)$.  Hence, for two curves $\alpha,\alpha' \in \mathcal{A}$ that are homotopic on $\mathcal{S}$, and for an endpoint $u$ of some curve $\beta \in \mathcal{B}'$, both $\alpha,\alpha'$ must lie above or below $u$. In other words, the curves in $\mathcal{A}$ can cross the curves in $\mathcal{B}'$ in at most $(z + 1)(2z+1) = O(z^2)$ different ways.  
 
In order to show that $\pi^{\ast}_{\mathcal{F}}(z) = O(z^4)$, consider any set $\mathcal{A}'\subset\mathcal{A}$ of $z$ double grounded $x$-monotone pseudo-segments in $\mathcal{A}$.  The arrangement of $\mathcal{A}'$ will partition the region between $\ell_1$ and $\ell_2$ into at most $O(z^2)$ cells.  The endpoints of each $x$-monotone pseudo-segment in $\mathcal{B}$ will lie in one of these cells.  Hence, for $\beta \in \mathcal{B}$, there are $O(z^4)$ ways to choose the cells that contain the endpoints of $\beta$.   Once the endpoints of $\beta$ have been chosen, we have determined exactly which curves among $\mathcal{A}'$ that $\beta$ crosses.  Indeed, by the $x$-monotone and pseudo-segment condition, we know exactly which curves in  $\mathcal{A}'$ lie above (below) the left endpoint of $\beta$, and which curves in  $\mathcal{A}'$ lie above (below) the right endpoint of $\beta$.   Hence, $\pi^{\ast}_{\mathcal{F}}(z) = O(z^4)$. \end{proof}

Let $\mathcal{A}$ be a collection of double grounded $x$-monotone pseudo-segments in the plane.  The \emph{vertical decomposition} of the arrangement of $\mathcal{A}$ is obtained by drawing a vertical segment from each crossing point and endpoint in the arrangement, in both directions, and extend it until it meets the arrangement of $\mathcal{A}$, else to $\pm \infty$.  Since $\mathcal{A}$ is double grounded, the grounds will appear in the vertical decomposition.  The vertical decomposition of $\mathcal{A}$ partitions the plane into cells called \emph{generalized trapezoids}, where each generalized trapezoid is bounded by at most two curves from $\mathcal{A}$ from above or below, and at most two vertical segments on the sides.  See Figure \ref{figvertical}. By applying standard random sampling arguments (e.g., see \cite{CS} or Lemma 4.6.1 in \cite{mat}), we obtain the following result known as the \emph{weak cutting lemma}.

\begin{figure}
\centering
\includegraphics[width=5cm]{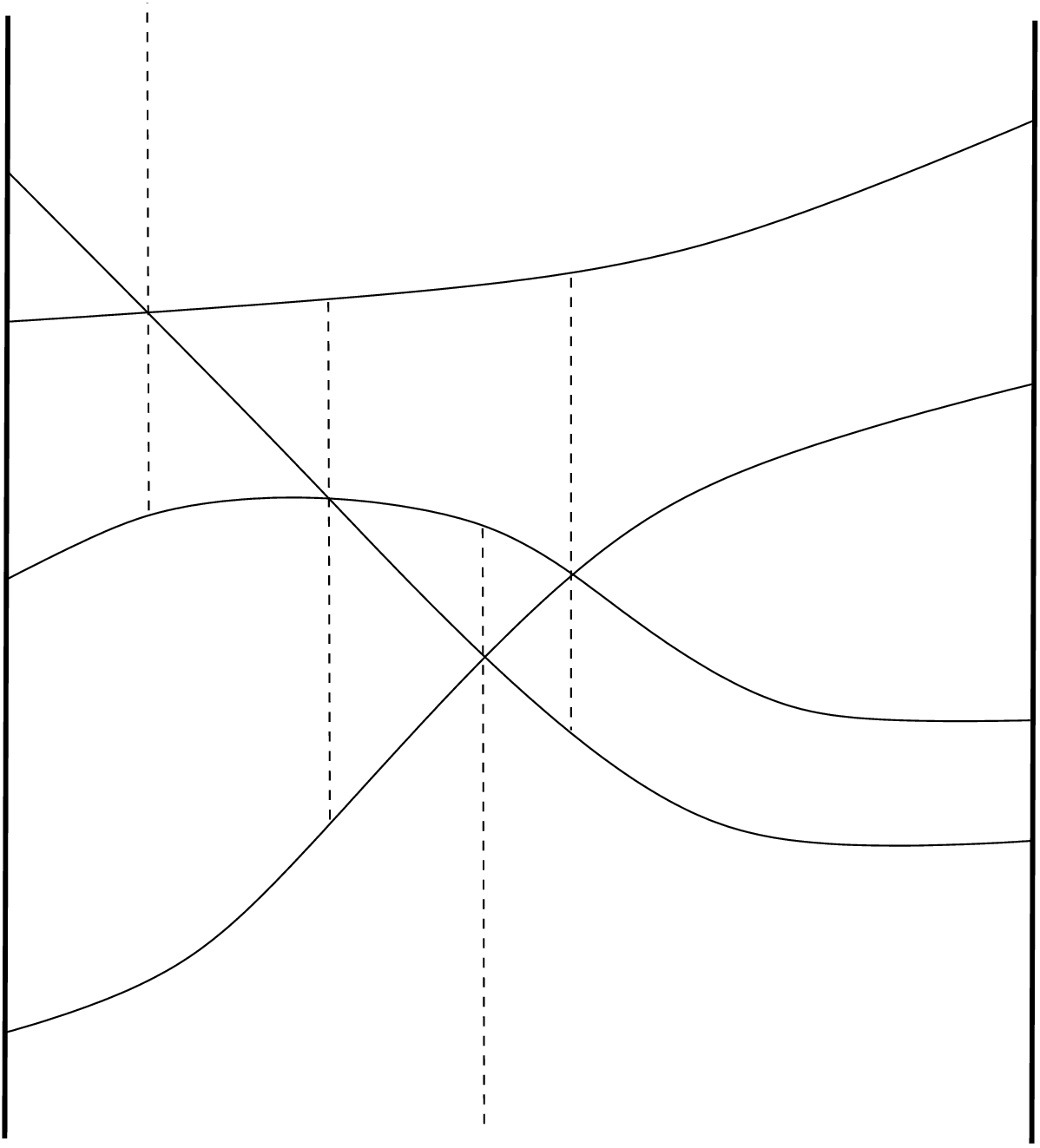}
\caption{Vertical decomposition of $\mathcal{A}$.}\label{figvertical}
\end{figure}
 
 \begin{lemma}[\cite{mat}, Lemma 4.6.1]\label{cut}
     Let $\mathcal{A}$ be a collection of $m$ double grounded $x$-monotone pseudo-segments in the plane.  Then for any parameter $r$, where $1\leq r \leq m$, there is a set of at most $s = 6r\log m$ curves in $\mathcal{A}$ whose vertical decomposition partitions the plane $\mathbb{R}^2 = \Delta_1\cup \cdots \cup \Delta_t$ into $t$ generalized trapezoids, such that $t = O(s^2)$, and the interior of each $\Delta_i$ crosses at most $m/r$ members in $\mathcal{A}$.  
 \end{lemma}

Let $f(m,n)$ denote the number of labelled intersection graphs between a collection $\mathcal{A}$ of $m$ double grounded $x$-monotone curves whose grounds are the vertical lines at $x = 0$ and $x = 1$, and a collection $\mathcal{B}$ of $n$ $x$-monotone curves whose endpoints lie inside the strip $S = [0,1]\times \mathbb{R}$ such that  $\mathcal{A}\cup \mathcal{B}$ is a collection of pseudo-segments.  We now prove the following.

\begin{lemma}\label{abpart}
For $m,n\geq 1$, we have

    $$f(m,n) \leq  2^{O(n^{2/3}m^{2/3}\log^2 m)} + 2^{O(n^{5/4}\log n)} + 2^{O(m\log^3m)}.$$
\end{lemma}

\begin{proof}
We can assume that $m,n$ are sufficiently large.  Let $\mathcal{A}\cup\mathcal{B}$ be a collection of pseudo-segments where $\mathcal{A}$ and $\mathcal{B}$ are as above. We naturally get two multiset systems.  Let $\mathcal{F}$ be the multiset system with ground set $\mathcal{B}$, and for each curve $\alpha \in \mathcal{A}$, we get a set in the multiset system consisting of the curves $\beta \in \mathcal{B}$ that intersect $\alpha$. The other multiset system, $\mathcal{F}^{\ast}$, is obtained by switching the roles of $\mathcal{A}$ and $\mathcal{B}$. By Lemma \ref{ptvc}, there is an absolute constant $c> 0$ such that $\pi_{\mathcal{F}}(z) \leq cz^2$ and $\pi_{\mathcal{F}^{\ast}}(z) \leq cz^4$.  By enumerating the possible multiset systems $\mathcal{F}$ we have $f(m,n) \leq h_{c,2}(m,n)$, and by enumerating the possible multiset systems $\mathcal{F}^{\ast}$ we have $f(m,n) \leq h_{c,4}(n,m)$.

Suppose $n > m^2$.  Then, by the first part of Theorem \ref{vc}, the number of intersection graphs between $\mathcal{A}$ and $\mathcal{B}$ is at most

\begin{equation}\label{eqnl1}
    h_{c,4}(n,m) \leq 2^{O(mn^{3/4}\log n)}\leq 2^{O(n^{5/4}\log n)}.
\end{equation} 

\noindent  If $n < m^{1/2}\log^2 m$, then by the first part of Theorem \ref{vc}, the number of intersection graphs between $\mathcal{A}$ and $\mathcal{B}$ is at most

\begin{equation}\label{eqnl2}
    h_{c,2}(m,n) \leq 2^{O(m^{1/2}n\log m)}\leq 2^{O(m\log^3 m)}.
\end{equation}
 
Let us assume that $m^{1/2}\log^2m < n < m^2$.   Set $r = \frac{n^{2/3}}{(m\log^{4} m)^{1/3}}$ and $s = 6r\log m$.  Since $m$ and $n$ are sufficiently large, we have $1 \leq r < s \leq m$.  For a set of $m$ double grounded $x$-monotone curves whose grounds are on the vertical lines $x  = 0$ and $x = 1$, there are $(m!)^2$ ways to order the left and right endpoints.  Let us fix such an ordering. 

Let $\mathcal{A}'\subset \mathcal{A}$ be a set of at most $s = 6r\log m$ curves from $\mathcal{A}$ whose arrangement gives rise to a vertical decomposition satisfying Lemma \ref{cut} with parameter $r$.  Note that there are at most $m^s$ choices for $\mathcal{A}'$, and by Lemma \ref{arrange}, there are at most $2^{O(s^2\log s)}$ arrangements for $\mathcal{A}'$, up to $x$-isomorphism.   Once the arrangement of $\mathcal{A}'$ is fixed, the vertical decomposition of $\mathcal{A}'$ is determined.

Let $\mathbb{R}^2 = \Delta_1\cup \cdots \cup \Delta_t$ be the vertical decomposition corresponding to  $\mathcal{A}'$, where $t = O(s^2)$.  Let $\mathcal{A}_i\subset \mathcal{A}$ be the curves in $\mathcal{A}$ that cross the cell $\Delta_i$.  For each curve $\alpha\in \mathcal{A'}$, by Lemma \ref{zone}, at most $O(s)$ vertical segments from the vertical decomposition have an endpoint on $\alpha$.  Moreover, at most $m$ curves from $\mathcal{A}$ cross $\alpha$.  Among these $O(s + m)$ points along $\alpha$, let us fix the order in which they appear along $\alpha$, from left to right.  Since there are at most $s^2$ vertical segments,  there are at most $(s^2 + m)^{O(s + m)} = m^{O(m)}$ ways to fix this ordering, and therefore, there are at most $m^{O(sm)}$ ways to fix such an ordering for each curve $\alpha \in \mathcal{A}'$.

Let $\beta \in \mathcal{B}$.  Then there are $O(s^4)$ choices for the cells $\Delta_i$ for which the endpoints of $\beta$ lie in.   Suppose that the left endpoint of $\beta$ lies in cell $\Delta_i$ and the right endpoint lies in $\Delta_j$, and consider the vertical lines $\ell_1$ and $\ell_2$ that go through the left and right endpoint of $\beta$ respectively.  Then for each $\alpha' \in \mathcal{A}\setminus(\mathcal{A}_i\cup \mathcal{A}_j)$, we have already determined if $\alpha'$ crosses $\beta$.  Indeed, let us consider the left endpoint of $\beta$ and the cell $\Delta_i$.  By the vertical decomposition, $\Delta_i$ is bounded either above or below by some curve $\alpha \in \mathcal{A}'$.  Without loss of generality, let us assume that $\Delta_i$ is bounded from above by $\alpha$.  Let $p$ be the point on $\alpha$ that intersects the left side of $\Delta_i$.  Then for any $\alpha' \in \mathcal{A}\setminus (\mathcal{A}_i\cup \mathcal{A}_j)$, we have already determined if the left endpoint of $\alpha'$ is above or below the left 
endpoint of $\alpha$ along the ground $x = 0$.  Moreover, we have already determined if $\alpha'$ crosses $\alpha$ to the left of point $p$.  Since $\alpha'$ does not cross $\Delta_i$, we have determined if $\alpha'$ crosses $\ell_1$ above or below $\beta$.  See Figure \ref{figalpha}.  By the same argument, we have determined if $\alpha'$ crosses $\ell_2$ above or below the right endpoint of $\beta$.  Therefore, by the pseudo-segment condition, we have determined if $\alpha'$ crosses $\beta$.

\begin{figure}
\centering
\includegraphics[width=4cm]{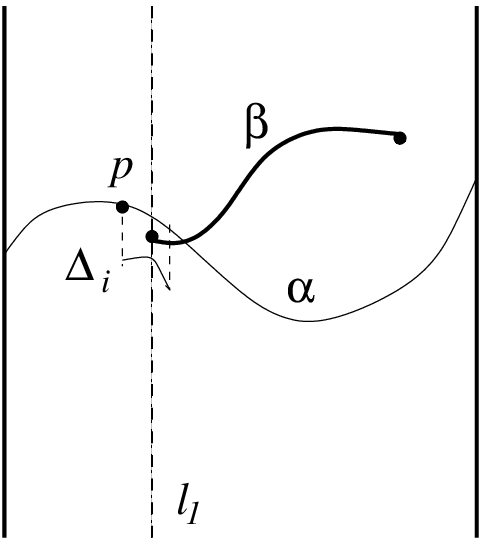}
\caption{Cell $\Delta_i$ bounded above by $\alpha$ and contains the left endpoint of $\beta$.}\label{figalpha}
\end{figure}

It remains to determine how many ways $\beta$ can cross the curves in $\mathcal{A}_i$ and $\mathcal{A}_j$.  By Lemma \ref{cut}, $|\mathcal{A}_i| \leq m/r$.   Let $\mathcal{B}_i$ denote the curves in $\mathcal{B}$ that have at least one endpoint in the cell $\Delta_i$.  Set $n_i  = |\mathcal{B}_i|$.  By Theorem \ref{vc}, there are at most

$$\begin{array}{ccl}
   h_{c,2}(|\mathcal{A}_i|,|\mathcal{B}_i|)  & \leq &  2^{O((m/r)^{1/2}n_i\log m)}
\end{array}$$

 \noindent ways the curves in $\mathcal{A}_i$ cross the curves in $\mathcal{B}_i$.  Putting everything together, the number of ways the curves in $\mathcal{A}$ cross the curves in $\mathcal{B}$ is at most

\begin{equation}\label{eqn14} (m!)^2m^s2^{O(s^2\log s)}m^{O(sm)}\left(s^4\right)^n\prod\limits_{i = 1}^t2^{O((m/r)^{1/2}n_i\log m)} .\end{equation}

\noindent Since $t = O(s^2)$, $r = \frac{n^{2/3}}{(m\log^{4} m)^{1/3}}$, and $s = 6r\log m \leq m$, we have

$$(m!)^2m^s2^{O(s^2\log s)}m^{O(sm)} \leq 2^{O(s^2(m/r)\log m)},$$

\noindent and

$$\left(s^4\right)^n\prod\limits_{i = 1}^t2^{O((m/r)^{1/2}n_i\log m)} \leq 2^{O( (m/r)^{1/2}n\log m)}.$$

\noindent Hence, (\ref{eqn14}) is at most

\begin{equation}\label{eqnl3}
    2^{O( (m/r)^{1/2}n\log m + s^2(m/r)\log m)}\leq 2^{O(n^{2/3}m^{2/3}\log^2 m)},
\end{equation}

\noindent  Combining (\ref{eqnl1}), (\ref{eqnl2}), and (\ref{eqnl3}), we have

  $$f(m,n) \leq  2^{O(n^{2/3}m^{2/3}\log^2 m)} + 2^{O(n^{5/4}\log n)} + 2^{O(m\log^3m)}.$$\end{proof}

 \noindent Hence, we have $f(n,n) \leq 2^{O\left(n^{4/3}\log^2 n\right)}.$
 
\begin{proof}[Proof of Theorem \ref{main12}]
    Let $g(n;p)$ be the number of labeled intersection graphs of at most $n$ $x$-monotone pseudo-segments in the vertical strip $[0,1]\times \mathbb{R}$, such that there are at most $p$ endpoints with $x$-coordinate in $(0,1)$.  Note that some pseudo-segments may contribute two endpoints to $p$.   Then we have the following recurrence.

\begin{claim}
    We have 

    $$g(n;p) \leq 2^{O\left(n^{4/3}\log^2n\right)}g^2(\lceil p/2\rceil;\lceil p/2\rceil).$$
\end{claim}

  \begin{proof}
      For $n$ $x$-monotone curves in the strip $S = [0,1]\times \mathbb{R}$, with $p$ endpoints in the interior of $S$, we can assume that these $p$ endpoints have distinct $x$-coordinates. We partition the interval $[0,1]$ into two parts $I_1,I_2$, so that the interior of each strip $S_i = I_i\times \mathbb{R}$ has at most $\lceil p/2\rceil$ endpoints. Next, we upper bound the number of labeled intersection graphs of the curves restricted to the strip $S_i$. Note that there are $n!$ ways to label the curves. 

Among the curves restricted to the strip $S_i$, let $\mathcal{A}_i$ denote the set of curves that go entirely through $S_i$, and let $\mathcal{B}_i$ be the set of curves with at least one endpoint in the interior of $S_i$.  There are at most $n!$ ways to determine the intersection graph among the curves in $\mathcal{A}_i$.  By Lemma \ref{abpart}, there are at most

$$f(|\mathcal{A}_i|,|\mathcal{B}_i|) \leq f(n,n) \leq 2^{O(n^{4/3}\log^2n)}$$

\noindent ways to determine the intersection graph between $\mathcal{A}_i$ and $\mathcal{B}_i$.  Finally, there are at most $g(\lceil p/2\rceil;\lceil p/2\rceil)$ ways to determine the intersection graph among the curves in $\mathcal{B}_i$.  Putting everything above together gives the desired recurrence.    \end{proof}

Since $p\leq 2n$, the recurrence above gives 

$$g(n;2n) \leq 2^{\sum\limits_{i  = 1}^{\log n} 2^iO\left((n/2^i)^{4/3}\log^2(n/2^i)\right)}g(1;1) \leq 2^{O(n^{4/3}\log^2n)}.$$

\end{proof}

\subsection{Intersection graphs with small clique number}\label{secbip1}

In this subsection, we prove Theorem \ref{main3}.

\begin{proof}[Proof of Theorem \ref{main3}]
 Let $g_k(n;p)$ be the number of labeled intersection graphs of at most $n$ $x$-monotone pseudosegments with clique number at most $k$ in the vertical strip $[0,1]\times \mathbb{R}$, such that there are at most $p$ endpoints with $x$-coordinate in $(0,1)$.   Similar to above, we will show 
 
$$g_k(n;p) \leq n^{6n + 2kp}g^2_k(\lceil p/2\rceil;\lceil p/2\rceil).$$

Indeed, for $n$ $x$-monotone pseudosegments in the strip $S = [0,1]\times \mathbb{R}$, with $p$ endpoints in the interior of $S$, we can assume that these $p$ endpoints have distinct $x$-coordinate. We partition the interval $[0,1]$ into two parts $I_1,I_2$, so that the interior of each strip $S_i = I_i\times \mathbb{R}$ has at most $ \lceil p/2\rceil$ endpoints.  We now bound the number of labeled intersection graphs of the curves restricted to $S_1$.

 There are at most $n!$ ways to label the curves in $S_1$.  There are at most $2^n$ ways to choose the set $\mathcal{A}$ of pseudo-segments that goes entirely through $S_1$.  Let $G_{\mathcal{A}}$ denote its intersection graph of $\mathcal{A}$.  Then $G_{\mathcal{A}}$ depends entirely on the permutation of the endpoints of $\mathcal{A}$. 
 Hence, $G_{\mathcal{A}}$ is an incomparability graph of a 2-dimensional poset and there are at most $n!$ ways to determine $G_{\mathcal{A}}$. Since $G_{\mathcal{A}}$ has clique number at most $k$, by Dilworth's theorem \cite{d}, $G_{\mathcal{A}}$ has chromatic number at most $k$.  Thus, there are at most $k^n$ ways to properly color the vertices of $G_{\mathcal{A}}$.  After fixing such a coloring, let $\mathcal{A}_1,\ldots, \mathcal{A}_k$ denote the color classes.  Since the curves in $\mathcal{A}_i$ are pairwise disjoint and go through $S_1$, for each curve $\gamma$ with an endpoint in the interior of $S_1$, there are at most $n^2$ ways $\gamma$ can intersect the curves in $\mathcal{A}_i$.  Moreover, since there are $p/2$ such endpoints, there are at most $p/2$ such curves. Therefore, there are at most $(n^2)^k$ ways $\gamma$ can intersect the curves in $\mathcal{A}$.   Since $k\leq n$, there are at most
 
 $$n!2^nn!k^n(n^2)^{kp/2}g_k(\lceil p/2;p/2\rceil) \leq n^{4n + kp}g_k(\lceil p/2\rceil ;\lceil p/2\rceil)$$
 
\noindent labeled intersection graphs among the curves restricted to $S_1$.  A similar argument holds for the curves restricted to $S_2$.  Hence,

$$g_k(n;p) \leq n^{8n + 2kp}g^2_k(\lceil p/2\rceil;\lceil p/2\rceil).$$

 \noindent Iterating the inequality above $t$ times gives

 $$g_k(n;p) \leq  n^{8n + 2kp}\left(\frac{p}{2}\right)^{8p + 2kp}\left(\frac{p}{2^2}\right)^{8p + 2kp}\cdots \left(\frac{p}{2^{t-1}}\right)^{8p + 2kp}g^{2^t}(\lceil p/2^t\rceil;\lceil p/2^t\rceil).$$
 
\noindent Hence for $t = \lceil\log_2n\rceil$, we have

$$g_k(n;p) \leq n^{8n + 2kp}p^{(8p + 2kp)t}.$$
 
\noindent By setting $p = 2n$, we have

$$g_k(n;2n) \leq 2^{O(kn\log^2n)},$$

\noindent and Theorem \ref{main3} follows. \end{proof}

\section{Bipartite intersection graphs of $x$-monotone curves}\label{secbip2}

In this section, we prove Theorem \ref{main}.  Recall that here, two $x$-monotone curves may cross each other more than once. The proof is very similar to the proof of Theorem \ref{main3} above.  Let $w(n;p)$ be the number of labeled bipartite intersection graphs of at most $n$ $x$-monotone curves in the vertical strip $[0,1]\times \mathbb{R}$, such that there are at most $p$ endpoints with $x$-coordinate in $(0,1)$.  We establish the following recurrence.



\begin{lemma}\label{rec}
We have

$$w(n;p) \leq n^{6n}w^2(\lceil p/2\rceil;\lceil p/2\rceil).$$

\end{lemma}

\begin{proof}

For $n$ $x$-monotone curves in the strip $S = [0,1]\times \mathbb{R}$, with $p$ endpoints in the interior of $S$, we can assume that these $p$ endpoints have distinct $x$-coordinates. We partition the interval $[0,1]$ into two parts $I_1,I_2$, so that the interior of each strip $S_i = I_i\times \mathbb{R}$ has at most $\lceil p/2\rceil$ endpoints. Next, we upper bound the number of labeled intersection graphs of the curves restricted to the strip $S_i$. Note that there are $n!$ ways to label the curves. 

For each curve $\gamma$, as the graph is bipartite, let us count the number of ways $\gamma$ can intersect the set of pairwise disjoint curves that go entirely through $S_i$.  By ordering these pairwise disjoint curves  vertically, this intersection set is an interval with respect to this vertical ordering.  Hence, $\gamma$ has at most $n^2$ ways to intersect the family of curves that goes entirely through $S_i$. This gives a total of at most $n!(n^2)^n < n^{3n}$ ways of determining the intersection graph in $S_i$, apart from the induced subgraph on the curves with at least one endpoint in the interior of $S_i$.  Since there are $p/2$ such endpoints, there are at most $p/2$ such curves. Thus we have at most $w(\lceil p/2\rceil;\lceil p/2\rceil)$ possible such intersection graphs of the curves with one end point in $S_i$. Thus we have at most $n^{3n}w(\lceil p/2\rceil,\lceil p/2\rceil)$ possible intersection graphs restricted to $S_i$.  Since the intersection graph of all $n$ curves is the union of the intersection graphs on $S_1$ and $S_2$, we get in total at most $(n^{3n}w(\lceil p/2\rceil;\lceil p/2\rceil))^2$ such choices.  \end{proof}

\begin{proof}[Proof of Theorem \ref{main}] It suffices to bound $w(n;2n)$ as the original $n$ curves have $2n$ endpoints.  Iterating the recurrence in Lemma \ref{rec} $t$ times gives

$$w(n;p) \leq n^{6n}\left(\frac{p}{2}\right)^{6p}\left(\frac{p}{2^2}\right)^{6p}\cdots \left(\frac{p}{2^{t-1}}\right)^{6p} w^{2^t}(\lceil p/2^{t }\rceil;\lceil p/2^{t}\rceil).$$

\noindent Thus for $t= \lceil \log_2 n\rceil $, we get   $$w(n;p) \leq n^{6n}p^{6pt}.$$

\noindent Hence,

$$w(n;2n)\leq 2^{O(n\log^2 n)}.$$

\end{proof}

Let us remark that in \cite{fpt}, the first two authors showed that there is an absolute constant $c> 0$ such that every $n$-vertex string graph with clique number $k$ has chromatic number at most $(C\frac{\log n}{\log k})^{c\log k}$. Together with Corollary~\ref{main2}, we obtain the following.

\begin{corollary}
For every $\epsilon>0$, there is $\delta>0$ such that the number of intersection graphs of $n$ $x$-monotone curves with clique number at most $n^{\delta}$ is at most $2^{n^{1+\epsilon}}$.
\end{corollary}

\section{Concluding remarks} 

An important motivation for enumerating intersection graphs of curves of various kinds came from a question in graph drawing~\cite{PT06}: How many ways can one draw a graph? The number of different (non-isomorphic) drawings of $K_n$, a complete graph of $n$ vertices, can be upper-bounded by the number intersection graphs of ${n\choose 2}$ curves. By~\cite{PT06}, this is at most $2^{(3/2^5+o(1))n^4}$. 

The number of non-isomorphic \emph{straight-line drawings} of $K_n$ cannot exceed the the number of different intersection graphs of ${n\choose 2}$ segments in the plane, which is $2^{(4+o(1))n^2\log n}$; see~\cite{S, PS}. However, the true order of magnitude of the number of straight-line drawings of $K_n$ is much smaller. As was pointed out in~\cite{PT06}, this quantity is equal to the number of \emph{order types} of $n$ points in general position in the plane. The latter quantity is $2^{(4+o(1))n\log n}$, according to seminal results of Goodman--Pollack~\cite{GP} and Alon~\cite{Al}, based on Warren's theorem in real algebraic geometry~\cite{W}. 

Recall that Theorem~\ref{vc} in Section \ref{secvc} shows that for $d \geq 2$ fixed and $m,n \geq 2$, the number $h'_d(m,n)$ of set systems of $m$ subsets of $[n]$ that have VC-dimension at most $d$ is at most $2^{O(nm^{1-1/d}\log m)}$. It would be interesting to remove the logarithmic factor in the exponent, which would answer the question of Alon et al.~\cite{alon} mentioned in the beginning of Section~\ref{secvc}. A natural approach, which has worked for similar enumerative problems, is to recast the problem as counting independent sets in an auxiliary hypergraph and use the hypergraph container method. Consider the $2^{d+1}$-uniform hypergraph $H$ with vertex set $2^{[n]}$ (so the vertices are just the subsets of $[n]$) and a $2^{d+1}$-tuple of vertices forms an edge if they shatter a subset of the ground set of size $d+1$. The function $h'_{d}(m,n)$ then just counts the number of independent sets of size $m$ in $H$. The hypergraph container method (introduced in \cite{BMS0,ST15}, see also \cite{BMS}) is a powerful tool that is useful for counting independent sets in similar settings. It would be interesting if one could adapt these techniques to give better bounds on $h'_d(m,n)$. 

The last five results in the introduction give upper bounds on the number of intersection graphs or the number of non-isomorphic drawings of graphs under various constraints. It would be interesting to close the gap between these upper bounds and lower bounds. 

The following simple construction shows that there are $2^{\Omega(n\log n)}$ unlabelled bipartite graphs on $n$ vertices that are intersection graphs of segments. Take one set (left group) of $k=n/(3\log n)$ vertical segments crossing the $x$-axis near -1, another set (middle group) of $k$ vertical segments crossing the $x$-axis near the origin, and another set (right group) of $k$ vertical segments crossing the $x$-axis near +1. For the middle group of vertical segments, we translate them vertically so that their upper-endpoints are very close to the origin, and form a ``staircase" pattern.  The remaining $n - 3k$ segments are horizontal.  For each horizontal segment $s$ and integer $i \in [k]$, we have the freedom of choosing the left endpoint of $s$ so that it intersects precisely the last $i$ vertical segments from the first group.  Likewise, for each $j\in [k]$, we have the freedom of choosing the right endpoint of $s$ so that it intersects precisely the first $j$ vertical segments from third group.  Moreover, by choosing the $y$-coordinate of segment $s$, for each $\ell [k]$, we can make $s$ intersect precisely the last $\ell$ segments from the middle group.  One gets $2^{(3-o(1))n\log_2 n}$ labelled bipartite intersection graphs (and hence at least $2^{(2-o(1))n \log_2 n}$ unlabelled bipartite intersection graphs). This shows that Theorem \ref{main} is tight up to a single logarithmic factor in the exponent. 

Viewing the same construction as a drawing of a matching (with the endpoints of segments as vertices of the matching), gives $2^{\Omega(n \log n)}$ non-isomorphic straight-line drawings of a matching on $n$ vertices whose edge-intersection graph is bipartite, providing a lower bound for Corollary \ref{drawcor}.



\medskip

\noindent \textbf{Acknowledgement.}  We would like to thank Zixiang Xu for pointing out the reference \cite{alon2} to us.  We would also like to thank the anonymous referees for helpful comments, and in particular, the referee that pointed out Theorem 1.7 in \cite{P} which led to the improved version of Theorem \ref{main12}.


\begin{thebibliography}{99}

\bibitem{Al} N.~Alon, The number of polytopes, configurations and real matroids, \emph{Mathematika} \textbf{33} (1986), 62--71.

\bibitem{alon} N. Alon, J. Balogh, B. Bollob\'as, and R. Morris, The structure of almost all graphs in a hereditary property, \emph{J. Combin. Theory Ser. B} \textbf{101} (2011), 85--110.


\bibitem{alon2} N. Alon, S. Moran, and A. Yehudayoff, Sign rank versus Vapnik-Chervonenkis dimension, \emph{Sb. Math.} \textbf{208} (2017), 1724--1757. 

\bibitem{BMS0} J. Balogh, R. Morris, and W. Samotij, Independent sets in hypergraphs, 
\emph{J. Amer. Math. Soc.} \textbf{28} (2015), 669-709.

\bibitem{BMS} J. Balogh, R. Morris, and W. Samotij, 
The method of hypergraph containers. In: \emph{Proceedings of the International Congress of Mathematicians-Rio de Janeiro 2018. Vol. IV. Invited lectures,} 3059--3092.
World Scientific Publishing Co. Pte. Ltd., Hackensack, NJ, 2018
.

\bibitem{bern} M. Bern, D. Eppstein, P. Plassmann, and F. Yao, Horizon theorems for lines and
polygons, In J.E. Goodman, R. Pollack, and W. Steiger, editors, \emph{Discrete and Computational Geometry: Papers from the DIMACS Special Year}, pages 45--66, AMS,
Providence, 1991.

\bibitem{CS} K.L. Clarkson and P. Shor, Applications of random sampling in computational geometry, II, \emph{Discrete Comput. Geom.} \textbf{4} (1989), 387--421.

\bibitem{d}  R. Dilworth, A decomposition theorem for partially ordered sets, \emph{Ann. of Math.} \textbf{51} (1950),
161--166.

\bibitem{FV} S. Felsner and P. Valtr, Coding and counting arrangements of pseudolines, \emph{Discrete Comput. Geom.} \textbf{46} (2011), 405--416.

\bibitem{fp} J. Fox and J. Pach, Applications of a new separator theorem for string graphs, \emph{Comb. Probab. Comput.} \textbf{23} (2014), 66--74.

\bibitem{fpt} J. Fox, J. Pach, and A. Suk, Quasiplanar graphs, string graphs, and the Erd\H{o}s-Gallai problem, \emph{European J. Combin.} {\bf 119} (2024), Paper No. 103811, 10 pp.

\bibitem{GP} J.~E.~Goodman and R.~Pollack, Upper bounds for configurations and polytopes in $\textbf{R}^d$, \emph{Discrete Comput. Geom.} \textbf{1} (1986), 219--227.

\bibitem{H} D. Haussler, Sphere packing numbers for subsets of the Boolean $n$-cube with bounded Vapnik Chervonenkis dimension, \emph{J. Combin. Theory Ser. A} \textbf{69} (1995), 217--232.

\bibitem{KP} B. Keszegh and D. P\'alv\"olgyi, The number of tangencies between two families of curves, \emph{Combinatorica} \textbf{43} (2023), 939--952.

\bibitem{K13} J. Kyn\v{c}l, Improved enumeration of simple topological graphs, \emph{Discrete Comput. Geom.} \textbf{10} (2013), 727--770.

\bibitem{mat} J. Matou\v{s}ek, \emph{Lectures on Discrete Geometry}, Springer--Verlag, New York, 2002.

\bibitem{MM} C. McDiarmid and T. M\"uller, The number of disk graphs, \emph{European J. Combin.} \textbf{35} (2014), 413--431.

\bibitem{PS} J. Pach and J. Solymosi, Crossing patterns of segments, \emph{J. Combin. Ser. A.} \textbf{96} (2001), 316--325.


\bibitem{PT06} J. Pach and G. T\'oth, How many ways can one draw a graph? \emph{Combinatorica} \textbf{26} (2006), 559--576.

\bibitem{P} P. Przytycki, Arcs intersecting at most once, \emph{Geom. Funct. Anal.} \textbf{25} (2015), 658--670.

\bibitem{Sa} N. Sauer, On the density of families of sets, {\em J. Combinat. Theory Ser. A} \textbf{13} (1972), 145--147.



\bibitem{S} L. Sauermann, On the speed of algebraically defined graph classes, \emph{Adv. Math.} \textbf{380} (2021), Article 107593, 55 pp.

\bibitem{ST15} D. Saxton and A. Thomason, Hypergraph containers,
\emph{Invent. Math.} \textbf{201} (2015), 925--992.


\bibitem{Sh} S. Shelah, A combinatorial problem, stability and order for models and theories in infinitary languages, {\em Pacific J.  Math.} \textbf{41} (1972), 247--261.



\bibitem{St} R. P. Stanley, On the number of reduced decompositions of elements of Coxeter groups,
\emph{European J. Combin.} \textbf{5} (1984), 359--372.

\bibitem{Valtr} P. Valtr,  Graph drawing with no k pairwise crossing edges, Graph drawing (Rome, 1997), 205--218. Lecture Notes in Comput. Sci., 1353
Springer--Verlag, Berlin, 1997.

\bibitem{W}  H. E. Warren, Lower bounds for approximation by linear manifolds, \emph{Trans.
Amer. Math. Soc.} \textbf{133} (1968), 167--178.

\end{thebibliography}
\end{document}